\documentclass[a4paper,12pt]{amsart}
\usepackage{graphicx}
\usepackage{amsmath,amssymb,amsthm,amsfonts}
\usepackage{amssymb}
\usepackage[active]{srcltx}
\usepackage{color}
\newtheorem{thm}{Theorem}[section]
\newtheorem{cor}[thm]{Corollary}

\newtheorem{prop}[thm]{Proposition}
\newtheorem{rem}[thm]{Remark}
\newtheorem{defn}[thm]{Definition}

\newcommand{\R}{\mathbb R}

\newcommand{\e}{\varepsilon}

\newcommand{\n}{\nabla}
\newcommand{\la}{\langle}
\newcommand{\ra}{\rangle}

\newcommand{\Om}{\Omega}

\newcommand{\D}{\hbox{D}}

\newcommand{\A}{\zeta}

\newcommand{\C}{\theta}

\begin{document}

\parindent 0pc
\parskip 6pt
\overfullrule=0pt

\title{Liouville-type results for some quasilinear anisotropic elliptic equations}

\author{Alberto Farina$^{\dagger}$, Berardino Sciunzi$^{*}$ and Domenico Vuono$^{*}$}
\address{$^{*}$Dipartimento di Matematica e Informatica, UNICAL, Ponte Pietro  Bucci 31B, 87036 Arcavacata di Rende, Cosenza, Italy}
\email{sciunzi@mat.unical.it, domenico.vuono@unical.it}
\address{$^{\dagger}$  Universit\'e de Picardie Jules Verne, LAMFA, CNRS UMR 7352, 33, rue Saint-Leu 80039 Amines, France}
\email{alberto.farina@u-picardie.fr}


\date{\today}

\thanks{The authors are members of INdAM.
    B. Sciunzi are partially supported by PRIN project 2017JPCAPN (Italy):
    {\em Qualitative and quantitative aspects of nonlinear PDEs.}}

\maketitle

\date{\today}

\begin{abstract}
We prove some Liouville-type theorems for stable solutions (and solutions stable outside a compact set)
of quasilinear anisotropic elliptic equations. Our results cover the particular case of the pure Finsler $p$-Laplacian.
\end{abstract}

\section{Introduction}

In this paper we study solutions, possibly unbounded and sign-changing, to some quasilinear anisotropic elliptic equations in general Euclidean domains and we obtain new Liouville-type theorems. More precisely, for any domain $\Omega$ of $\R^N$, $N\geq 2$, and $q>1$
we consider the quasilinear elliptic equation:
\begin{equation}\label{eq:Euler-Lagrange}
-\Delta_B^H u=|u|^{q-1}u
\quad \text{in } \Omega\,,
\end{equation}
 
where $\Delta_B^H$ is the anisotropic operator \begin{equation}\label{operatoranisotropic}
\Delta_B^H u=\operatorname{div}(B'(H(\nabla u))\nabla H(\nabla u)),
\end{equation}
$H$ is a Finsler norm and $B$ satisfies some natural growth conditions defined below. Operators of the form \eqref{operatoranisotropic} naturally arise in the description of several anisotropic phenomena
in  material science
\cite{CaHo, Gu}, biology \cite{AnInMa} and in image processing \cite{EsOs,
PeMa}.

If $H$ is the standard Euclidean norm and $B(t)=\frac{t^p}{p}$, $p>1$, the operator $\Delta_B^H$ reduces to the classical $p-$Laplacian operator. In this case, Liouville-type results for non-negative solutions and subcritical values of the exponent $q$, have been obtained in the celebrated papers \cite{GS1,GS2} for $p=2$ and  \cite{SeZo} for $p\neq 2$. The issue of the study of Liouville-type theorems for  stable solutions, and for solutions stable outside a compact set, possibly unbounded and changing-sign, was started in \cite{Farina1, Farina2} for the semilinear case $p=2$ (see also \cite{DF}). The quasilinear case $p\neq 2$  has been developed in \cite{DFSV} (see also \cite{FWM}).

Here we deal with the anisotropic case, i.e., the general case in which $H$ is a Finsler norm (see the assumption $(h_H)$ at the end of this section) and $B$ satisfies
\begin{itemize}
    \item[(h$_B$)]
    \begin{enumerate}
        \item[(i)] $B \in C^2((0,+\infty))$;

        \item[(ii)] $B(0)=B'(0)=0,\quad B(t), B'(t), B''(t)>0 \quad \forall t \in (0,+\infty)$;

        \item[(iii)] there exist $p>1$, $\kappa \in [0,1]$, $\gamma > 0$, $\Gamma > 0$ such that
        \begin{equation}\label{hp:boundAboveBelow}
            \begin{split}
                \gamma(\kappa+t)^{p-2}t \leq & B'(t) \leq \Gamma (\kappa+t)^{p-2}t \quad \forall\,t>0\\
                \gamma(p-1)(\kappa+t)^{p-2} \leq & B''(t) \leq \Gamma (p-1) (\kappa+t)^{p-2}\quad \forall\,t>0\,.
            \end{split}
        \end{equation}
    \end{enumerate}
\end{itemize}

For $p>1$, a function $u \in W^{1,p}_{loc}(\Omega )$ is a weak solution of
\eqref{eq:Euler-Lagrange} if
\begin{equation}\label{debil1}
\int_\Omega B'(H(\nabla u)) \langle \nabla H(\nabla u), \nabla
\varphi \rangle\,dx\,=\,\int_\Omega |u|^{q-1}u\varphi\,dx,
\quad\forall \varphi\in C^1_c(\Omega)\,.
\end{equation}

In \cite{CFV} the authors proved that locally bounded weak solutions $u$ to \eqref{eq:Euler-Lagrange} are actually of class $C^{1,\alpha}_{loc}(\Omega)\cap C^2(\Omega \setminus Z)$,
where $Z$ is the set of critical points of $u$ and $\alpha\in (0,1)$ (as in the isotropic case). In what follows, we shall use repeatedly those regularity properties without recalling them explicitly. Moreover, in \cite{AGF, CaRiSc}, new results are given about the
integrability of the second weak derivatives of the solutions.

Let us recall the following:

\begin{defn}[see \cite{DS1}]
Assume $p\ge 2$. Given a domain $\Omega\subset \R^N$, $\rho\in L^1(\Omega)$, $\rho \ge 0$ and $1\le p<\infty$, we define the space $H_\rho^{1,p}(\Omega)$ as the closure  of $C^1(\Omega)$ with respect to the norm 
\begin{equation}\label{normaH^1}
    \|v\|_{H_\rho^{1,p}(\Omega)}=\left(\int_\Omega|v|^p\right)^{1/p}+\left(\int_\Omega\rho|\nabla v|^p\right)^{1/p}.
\end{equation}
Similarly, the space $H_{0,\rho}^{1,p}(\Omega)$ is defined as the closure of $C_c^1(\Omega)$ in $H_{\rho}^{1,p}(\Omega)$.  \\ The linearized operator of \eqref{eq:Euler-Lagrange} is defined by  
\begin{equation}
\begin{split}
    L_u(v,\varphi)=\int_\Omega B''(H(\nabla u))  \langle \nabla H(\nabla
u), \nabla v \rangle\langle \nabla H(\nabla
u), \nabla \varphi \rangle+ \\
\int_{\Omega}B'(H(\nabla
u))\langle D^2H(\nabla u)\nabla v,\nabla \varphi \rangle
-q|u|^{q-1} v \varphi ,
\end{split}
\end{equation}

for every $v\in H^{1,2}_{\rho,loc}(\Omega)$ and $\varphi \in C_c^1(\Omega)$, where $\rho=H(\nabla u)^{p-2}\in L^1_{loc} (\Omega)$, for $p\ge 2$. The choice of the weight $\rho$ is justified by assumptions $(h_B)$. By density argument we can consider $\varphi \in H^{1,2}_{0,\rho}(\Omega)$.

\end{defn}

\begin{defn}
We say that a weak solution $u$ is stable if
\begin{equation}\label{condizionestability}
L_u(\varphi,\varphi)\ge 0 \qquad\forall \varphi \in C^1_c(\Omega).
\end{equation}

Moreover, we say that a weak solution is stable outside a compact set $K\subset \Omega$ if \eqref{condizionestability} holds for every $\varphi\in C^1_c(\Omega \setminus K).$
\end{defn}

Our main result is stated in the following theorem.

\begin{thm}\label{teoremaliouville1}

Let us assume that $(h_H)$ and $(h_B)$ hold and let $u$ be a weak stable solution of \eqref{eq:Euler-Lagrange} with $p\ge 2$. If one of the following holds 
\begin{itemize}
    \item [(i)] $\Gamma\kappa ^{p-2}(p-1)\ge 1$,  $q> \Gamma^2\kappa ^{p-2}\gamma^{-1}(p-1)^2$ and  
    $$N<g(q):=2\frac{\Gamma^2\kappa^{p-2}\gamma^{-1}(p-1)^2(q-1)+2q+2\sqrt{q(q-\Gamma^2\kappa^{p-2}\gamma^{-1}(p-1)^2)}}{\Gamma^2\kappa^{p-2}\gamma^{-1}(p-1)^2(q-1)}$$
    \item [(ii)]  $0<\Gamma\kappa^{p-2}(p-1)<1$, $q> \kappa ^{2-p}\gamma^{-1}$ and $$N<f(q):=2\frac{\kappa^{2-p}\gamma^{-1}(q-1)+2q+2\sqrt{q(q-\kappa^{2-p}\gamma^{-1})}}{\kappa^{2-p}\gamma^{-1}(q-1)}$$
    \item[(iii)] $\kappa=0$, $q> \Gamma\gamma^{-1}(p-1)$ and  
    $$N<
s(q):=p\frac{\Gamma\gamma^{-1}(p-1)(q-1)+2q+2\sqrt{q(q-\Gamma\gamma^{-1}(p-1))}}{\Gamma\gamma^{-1}(p-1)(q-(p-1))}
$$
\end{itemize}

Then $u\equiv 0$.
\end{thm}

We remark that under our structural assumptions, it also follows that $q>(p-1)\ge 1$.

Our result applies, in particular, to the case of the pure Finsler $p$-Laplacian, namely             \\
$B(t)=\frac{t^p}{p}$. In this case we have the following:

\begin{cor}\label{corollario}
Let us assume that $(h_H)$ holds and let $u$ be a weak stable solution of 
\begin{equation}
    -\operatorname{div}(H(\nabla u)^{p-1}\nabla H(\nabla u))=|u|^{q-1}u \qquad \text{in }\R^N.
\end{equation}
Assume that $p\ge 2$ and 

\begin{equation}
			\left\{\begin{array}{llll}
			p-1<q<\infty	 \qquad &\text{if } N\le  \frac{p(p+3)}{p-1}\\ \\
				p-1<q<q_c(N,p) \qquad &\text{if } N>\frac{p(p+3)}{p-1}\\
				
			\end{array}\right.
	\end{equation}
with
\begin{equation}
    q_c(N,p)=\frac{\left((p-1)N-p\right)^2+p^2(p-2)-p^2(p-1)N+2p^2\sqrt{(p-1)(N-1)}}{(N-p)\left((p-1)N-p(p+3)\right)}.
\end{equation}
Then, $u\equiv 0$.
\end{cor}
\begin{rem}
	Let us stress the fact that
the critical exponent $q_c(N,p)$ is the one found in \cite{DFSV, Farina2} when $H$ is the standard Euclidean norm. Moreover, our result is sharp, see \cite{DFSV, Farina2}.
\end{rem}

Exploiting the Pohozaev identity in the anisotropic setting, see \cite{MSP}, we shall also prove a Liouville-type result for solutions that are stable outside a compact set. 

\begin{thm}\label{teoremaliouville2}
Let us assume that $(h_H)$ and $(h_B)$ hold, with $\kappa=0$ in $(h_B)$.
Let $u$ be a weak solution of \eqref{eq:Euler-Lagrange} which is stable outside a compact set. Assume $p\ge 2$ and 
\begin{equation}\label{assunzioni}
			\left\{\begin{array}{llll}
			\Gamma\gamma^{-1}(p-1)<q<\infty	 \qquad &\text{if } N\le  p\gamma\Gamma ^{-1}\\ \\
				\Gamma\gamma^{-1}(p-1)<q<\frac{N\left(p-\frac{\Gamma}{\gamma}\right)+p}{N\frac{\Gamma}{\gamma}-p}, \qquad &\text{if } N> p\gamma\Gamma^{-1}.\\  
				
			\end{array}\right.
	\end{equation}
	
	Then $u\equiv 0$.
	
\end{thm}

    Although the case $k\neq 0$ is non degenerate, in our framework it turns out to be quite involved, since the techniques developed in \cite{DFSV,Farina2} do not work in this case. This is due to the fact that, when $k$ is not zero, we have a loss of the homogeneity proprieties of the stress field $B'(H(\nabla u))\nabla H(\nabla u)$. Indeed, in this case, $B'(t)\sim t^{p-1}$ for $t$ large, while $B'(t)\sim t$ for $t$ small. We also note that the parameter $k$ plays a role in the statement of our Liouville-type theorem.

 We conclude the introduction by giving the assumptions on the norm $H$. \\ 
 
 A function $H$ is said a Finsler norm if:
\begin{itemize}
    \item[(h$_H$)]
\begin{enumerate}
    \item[(i)] $H$ is a norm of class $C^2(\R^N\setminus \{0\})$;

    \item[(ii)] the set $\mathcal{B}_1^H:=\{\xi \in \R^N  :  H(\xi) < 1\}$ is \emph{uniformly convex}, i.e., all the principal curvatures of its boundary are bounded away from zero. \,
\end{enumerate}
\end{itemize}

\begin{rem}\label{RMK1}\rm

We recall that (see \cite{CFV2,CFV}) assumption (ii) in (h$_H$) is equivalent to the following property:
\begin{equation}
\exists \Lambda > 0: \quad \langle D^2H(\xi)v, v \rangle \geq \Lambda |v|^2 \quad \forall \xi \in \partial \mathcal{B}_1^H, \; \forall v \in \nabla H(\xi)^\bot
\end{equation}
and it is equivalent to assume that $\D^2 (H^2)$ is positive definite.
\end{rem}

Since $H$ is a norm in $\R^N$, we immediately get that there exist $c_1,c_2>0$ such that:
\begin{equation}\label{H equiv euclidea}
c_1|\xi|\leq H(\xi)\leq c_2|\xi|,\,\quad\forall\, \xi\in\R^N.
\end{equation}

Moreover, since $H$ is differentiable in $\R^N \setminus \{0\}$ and $1$-homogeneous,
the following identities hold true:
\begin{equation}\label{eulero}
\la\n H(\xi),\xi\ra=H(\xi) \qquad\forall\, \xi \in \R^N\setminus \{0\},
\end{equation}

\begin{equation}\label{grad 0 omog}
\nabla H(t\xi)=\hbox{sign}(t)\n H(\xi) \qquad\forall\,\xi \in \R^N\setminus \{0\},\, \forall t\neq 0,
\end{equation}

\begin{equation}\label{hess -1 omog}
\D^2 H(t\xi)=\frac{1}{|t|}\D^2H(\xi) \qquad\forall\,\xi \in  \R^N\setminus \{0\},\, \forall t\neq 0,
\end{equation}

\begin{equation}\label{hess nullo}
\D^2H(\xi)\xi=0 \qquad\forall\,\xi\in\R^N\setminus \{0\}.
\end{equation}
Finally, by \eqref{grad 0 omog} it follows that
there exists $M>0$ such that \begin{equation}\label{grad H limitato}
    |\nabla H(\xi)| \le  M \qquad \forall \xi\in\R^N\setminus \{0\}.
\end{equation}

\section{Main results}

The main tool for the proof of Theorems \ref{teoremaliouville1}
is the inequality stated in the following proposition.

\begin{prop}\label{propimportante}
Let $\Omega$ be a domain in $\R^N$, with $N\ge 2$, and $u$ be a weak stable solution of \eqref{eq:Euler-Lagrange}
with $p\ge 2$.
Then, if $\Gamma\kappa^{p-2}(p-1)\ge 1$ and $q>\Gamma ^2\kappa^{p-2}\gamma ^{-1}(p-1)^2$, for any
        \begin{equation}\label{valoribeta}
            \beta \in \left[1, \frac{2q-\Gamma ^2\kappa^{p-2}\gamma^{-1}(p-1)^2+2\sqrt{q(q-\Gamma ^2\kappa^{p-2}\gamma ^{-1}(p-1)^2)}}{\Gamma ^2\kappa^{p-2}\gamma ^{-1}(p-1)^2}\right)
        \end{equation}

and for any integer
\begin{equation}\label{assunzionel}
    l>\max\left\{ \frac{q+\beta}{q-(p-1)},2\right\}
\end{equation}
there exist a positive constant
$C=C(B,H,\delta,\e,q,\beta)$ such that
\begin{equation}\label{propfondamentale}
     \int_\Omega |u|^{q+\beta}\varphi ^{pl}\le C \int_{\Omega_1}|\nabla \varphi| ^{2\frac {q+\beta}{q-1}}+ C\int_{\Omega_2}|\nabla \varphi| ^{p\frac {q+\beta}{q-p+1}}\quad \forall \varphi\in C^1_c(\Omega),\quad 0\le\varphi\le 1,
\end{equation}
 where $\Omega_1:=\Omega_1(\delta):=\{x\in \Omega \,:\, H(\nabla u(x)) <\delta \}$ and $\Omega _2:=\Omega_2(\delta):= \Omega \setminus \Omega _1$, with $\delta >0$ fixed.

\

On the other hand, if $0<\Gamma\kappa^{p-2}(p-1)<1$ and $q>\kappa ^{-p+2}\gamma ^{-1}$, for any
    
\begin{equation}\label{valoribeta1}
          \beta \in
    \left[1, \frac{2q-\kappa^{-p+2}\gamma^{-1}+2\sqrt{q(q-\kappa^{-p+2}\gamma ^{-1})}}{\kappa^{-p+2}\gamma ^{-1}}\right)
    \end{equation}
   and for any integer $l$ as in \eqref{assunzionel}, the inequality \eqref{propfondamentale} still holds true.
\end{prop}

\begin{rem}
The same result holds true if $u$ is stable outside a compact set $K\subset \Omega$. In this case the test functions $\varphi$ belong to $C^1_c(\Omega\setminus K)$.
\end{rem}

\proof

As already pointed out after Theorem \ref{teoremaliouville2}, the case $\kappa>0$ is quite involved. To overcome these difficulties we shall split the domain $\Omega$ into two subdomains, $\Omega_1$ and $\Omega_2$. We split the proof into several steps.

\textbf{Step 1.} \emph{For any $\beta \ge 1$ and $\e>0$ small enough there exists a positive constant
$\bar{C}=\bar{C}(B,H,\delta,\e)$ such that
\begin{equation}\label{step1}
\begin{split}
    \min \{\beta \gamma \kappa ^{p-2}-\e^2,\frac{\beta \gamma \Gamma^{-1}}{p-1}-\e ^2\} &\left(\int_{\Omega_1} H(\nabla u)^2 |u|^{\beta-1} \varphi ^p +  \int_{\Omega_2} B''(H(\nabla u)H(\nabla u)^2 |u|^{\beta-1} \varphi ^p\right) \\ \le \bar{C} &\int_{\Omega_1} |\nabla \varphi|^2 |u|^{\beta +1}\varphi ^{p-2}+ \bar{C}\int_{\Omega_2} |\nabla \varphi|^p |u|^{\beta+p-1}+\int_{\Omega} |u|^{q+\beta} \varphi ^p
    \end{split}
\end{equation}
for any non negative $\varphi \in C^1_c(\Omega)$.}

Let us consider
$$
\phi =|u|^{\beta-1}u\varphi^p\in C^1_c(\Omega),
$$
so that
$\nabla \phi =p |u|^{\beta-1}u\varphi ^{p-1}\nabla\varphi+\beta |u|^{\beta-1}\varphi ^p \nabla u$.
Taking $\phi$ as test function in \eqref{debil1}, we get:
\begin{equation}\label{eqdistep1}
\begin{split}
\beta&\int_\Omega B'(H(\nabla u)) \langle \nabla H(\nabla u), \nabla
u \rangle |u|^{\beta -1}\varphi ^p\\&=-p\int_\Omega B'(H(\nabla u))
\langle \nabla H(\nabla u), \nabla \varphi \rangle
|u|^{\beta-1}u\varphi^{p-1}+ \int_{\Omega}|u|^{q+\beta} \varphi ^p.
\end{split}
\end{equation}
We emphisize  that, when the gradient vanishes, the following computations are well defined thanks to the structural hypothesis on $B(\cdot).$

Combining \eqref{eulero}, \eqref{grad H limitato} and $(h_B)-(iii)$ with
\eqref{eqdistep1} we get:
\begin{equation}\label{eq}
    \begin{split}
&\beta\int_\Omega B'(H(\nabla u))  H(\nabla u)|u|^{\beta -1}\varphi
^p \\ &\le Mp\int_\Omega B'(H(\nabla u)) |\nabla \varphi|
|u|^{\beta}\varphi^{p-1}+ \int_{\Omega}|u|^{q+\beta} \varphi ^p.
\end{split}
\end{equation}

Under assumptions $(h_B)-(iii)$, there exist positive constants
$C_1=C_1(B,\delta)=\Gamma(\kappa+\delta)^{p-2}$, $\overline{C}_1=\overline{C}_1(B,\delta)=\Gamma(p-1)(\kappa+\delta)^{p-2}$, $C_2=C_2(B,\delta)$ and  
$\overline C_2=\overline C_2(B,\delta)$ such that
\begin{equation}\label{disuguglianzesuB'1}
\begin{split}
     B'(H(\nabla u))&\le C_1 H(\nabla u) \quad\quad \text{in } \Omega_1\\
        B'(H(\nabla u))&\le C_2 H(\nabla u)^{p-1} \quad \text{in } \Omega_2 \\ B''(H(\nabla u))&\le \overline{C}_1  \quad\quad\quad \quad \quad \text{ in } \Omega_1\\
        B''(H(\nabla u))&\le \overline{C}_2 H(\nabla u)^{p-2} \quad \text{in } \Omega_2\,.
    \end{split}
\end{equation}
 Moreover, we have that
 \begin{equation}\label{disuguaglianzesuB'2}
     \begin{split}
        \gamma \kappa ^{p-2}H(\nabla u)\le B'(H(\nabla u)) \quad \text{in } \Omega \\
         B''(H(\nabla u))H(\nabla u)\le \Gamma\gamma ^{-1}(p-1)B'(H(\nabla u)) \quad \text{in }\Omega\,.
     \end{split}
 \end{equation}

Using \eqref{disuguglianzesuB'1} and \eqref{disuguaglianzesuB'2} in
\eqref{eq} we get:

\begin{equation}\label{totequazione}
\begin{split}
\beta\gamma\kappa ^{p-2}&\int_{\Omega_1} H(\nabla u)^2|u|^{\beta-1}\varphi ^p +
\beta\frac{\gamma\Gamma^{-1}}{p-1}\int_{\Omega_2} B''(H(\n u))H(\nabla u)^2 |u|^{\beta-1}\varphi ^p\\ \le
MpC_1&\int_{\Omega_1} H(\nabla u) |\nabla \varphi|
|u|^{\beta}\varphi^{p-1}+ MpC_2\int_{\Omega_2} H(\nabla u)^{p-1}
|\nabla \varphi|  |u|^{\beta}\varphi^{p-1}+
\int_{\Omega}|u|^{q+\beta} \varphi ^p.
\end{split}
\end{equation}

Now we estimate the first two terms of the right side of \eqref{totequazione}. By the standard weighted Young's inequality we get:
\begin{equation}\label{young2}
\int_{\Omega_1} H(\nabla u) |\nabla \varphi|
|u|^{\beta}\varphi^{p-1}
  \le \e^2\int_{\Omega_1} H(\nabla u)^2|u|^{\beta -1}\varphi ^p+\bar{C''}_{\e}\int_{\Omega_1} |\nabla \varphi|^2|u|^{\beta+1}\varphi^{p-2}
\end{equation}
where $\bar{C''}_{\e}$ is a positive constant depending on $\e$.

Using again the weighted Young's inequality with exponents $p$ and $\frac{p}{p-1}$ and
recalling that
\begin{equation}\label{disuguaglianzeB''}
    \gamma (p-1)H(\nabla u)^p\le B''(H(\nabla u)) H(\nabla u)^2 \quad \text{in } \Omega
\end{equation}
we estimate the second term of \eqref{totequazione}:
\begin{equation}\label{young1}
\begin{split}
&\int_{\Omega_2} H(\nabla u)^{p-1} |\nabla \varphi|
|u|^{\beta}\varphi^{p-1}\\&\le \e^2\gamma(p-1)\int_{\Omega_2}
H(\nabla u)^p |u|^{\beta-1}\varphi
^p+\bar{C'}_{\e} \int_{\Omega_2}|\nabla \varphi|^p
|u|^{\beta+p-1}\\&\le \e^2 \int_{\Omega_2}
B''(H(\nabla u))H(\nabla u)^2 |u|^{\beta-1}\varphi
^p+\bar{C'}_{\e} \int_{\Omega_2}|\nabla \varphi|^p
|u|^{\beta+p-1},
\end{split}
\end{equation}
where $\bar{C'}_{\e}$ is a positive constant depending on $\e.$
By \eqref{young1}, \eqref{young2} and \eqref{totequazione} we have

\begin{equation}
\begin{split}
    \min \{\beta \gamma \kappa ^{p-2}-\e^2,\frac{\beta \gamma \Gamma^{-1}}{p-1}-\e ^2\} &\left(\int_{\Omega_1} H(\nabla u)^2 |u|^{\beta-1} \varphi ^p +  \int_{\Omega_2} B''(H(\nabla u)H(\nabla u)^2 |u|^{\beta-1} \varphi ^p\right) \\ \le \bar{C} &\int_{\Omega_1} |\nabla \varphi|^2 |u|^{\beta +1}\varphi ^{p-2}+ \bar{C}\int_{\Omega_2} |\nabla \varphi|^p |u|^{\beta+p-1}+\int_{\Omega} |u|^{q+\beta} \varphi ^p
    \end{split}
\end{equation}
for any non negative $\varphi \in C^1_c(\Omega)$. Thus we have proved \eqref{step1}.

\textbf{Step 2.} \emph{Let us fix $\delta,\epsilon>0$  and set 
\begin{equation}
\alpha=\alpha(\epsilon,\delta)=q-\frac{\max \left\{ \left(\frac{\beta+1}{2}\right)^2\Gamma (\kappa+\delta)^{p-2}+\epsilon^2,\left(\frac{\beta+1}{2}\right)^2+\epsilon^2\right\}}{\min \left\{\beta \gamma \kappa ^{p-2}-\epsilon^2,\frac{\beta \gamma \Gamma^{-1}}{p-1}-\epsilon ^2\right\}}.
    \end{equation}
Then there exist a positive constant $\tilde{C}=\tilde{C}(B,H,\delta,\epsilon,\beta)$ such that
\begin{equation}\label{step2}
\alpha\int_{\Omega} |u|^{q+\beta}\varphi^p\le \tilde{C}\int_{\Omega_1} |\nabla \varphi|^2|u|^{\beta+1}\varphi^{p-2}+\tilde{C}\int_{\Omega_2}|\nabla \varphi|^p |u|^{\beta+p-1}
\end{equation}
for any non-negative $\varphi\in C^1_c(\Omega)$.
}
\\

To prove the claim, let us consider 
$$
\psi=|u|^{\frac{\beta-1}{2}}u\varphi
^{\frac{p}{2}}\in C^1_c(\Omega)\,,
$$
so that $\nabla \psi=
\left(\frac{\beta+1}{2}\right)
\varphi^{\frac{p}{2}}|u|^{\frac{\beta-1}{2}}\nabla
u+\frac{p}{2}|u|^{\frac{\beta-1}{2}}u\varphi ^{\frac{p}{2}-1}\nabla
\varphi.$ Recalling \eqref{hess nullo}, we have  \begin{equation}
    \int_{\Omega}B'(H(\nabla
u))\langle D^2H(\nabla u)\nabla u,\nabla \psi \rangle=\int_{\Omega}B'(H(\nabla
u))\langle D^2H(\nabla u)\nabla \psi,\nabla u \rangle=0
\end{equation} and therefore substituting $\psi$ in the stability condition
\eqref{condizionestability} and recalling the Euler's theorem \eqref{eulero}
we obtain:
\begin{eqnarray}\label{stimastabile}
\nonumber q\int_{\Omega} |u|^{q+\beta}\varphi ^p &\le&
  \left(\frac{\beta+1}{2}\right)^2\int_{\Omega} B''(H(\nabla u))H(\nabla u)^2|u|^{\beta-1}\varphi^p \\
\nonumber &+&\frac{p^2}{4}\int_\Omega B''(H(\nabla u))|\nabla H(\nabla u)|^2|\nabla \varphi|^2|u|^{\beta+1}\varphi^{p-2} \\
\nonumber &+&\left( \frac{\beta+1}{2}\right)p\int_\Omega B''(H(\nabla u))H(\nabla u)|\nabla \varphi||\nabla H(\nabla u)||u|^{\beta}\varphi^{p-1}\\ \nonumber &+&\frac{p^2}{4}\int_\Omega B'(H(\nabla u)) \langle D^2H(\nabla u)\nabla \varphi,\nabla \varphi\rangle |u|^{\beta+1}\varphi^{p-2}\\  
&=&I_1+I_2+I_3+I_4\,,
\end{eqnarray}
By \eqref{disuguglianzesuB'1} and recalling that there exists $M>0$ such that $|\nabla H(\xi)| \le  M$ for every $\xi\in\R^N\setminus \{0\}$, we have 

\begin{equation}
    I_2\le
   \frac{p^2}{4}M^2\overline{C}_1\int_{\Omega_1} |\nabla \varphi|^2|u|^{\beta+1}\varphi^{p-2}+
   \frac{p^2}{4}M^2\overline{C}_2\int_{\Omega_2} H(\nabla u)^{p-2}|\nabla \varphi|^2|u|^{\beta+1}\varphi ^{p-2}
\end{equation}
For\footnote{In the case $p=2$, the estimate \eqref{stimaI_2} follows immediately. Indeed in this case we do not need to use Young's inequality.} $p>2$, we use Young's inequality with exponents
$\frac {p}{2}$ and $\frac{p}{p-2}$ and \eqref{disuguaglianzeB''} to
get:
\begin{equation}\label{stimaI_2}
\begin{split}
   I_2\le
   \frac{p^2}{4}M^2\overline C_1\int_{\Omega_1} |\nabla \varphi|^2|u|^{\beta+1}\varphi^{p-2}
   &+\frac{\e ^2}{3}\int_{\Omega_2} B''(H(\nabla u))H(\nabla u)^{2}|u|^{\beta-1}\varphi ^p\\
   &+\tilde{C}_2\int_{\Omega_2}|\nabla \varphi|^p|u|^{\beta+p-1}\,,
\end{split}
\end{equation}

where $\tilde{C}_2=\tilde{C}_2(B,H,\delta,\e)$ is a
positive constant.

Now we estimate $I_4$. By \eqref{hess -1 omog} there exists $M'>0$ such that $|D^2H(\xi)|\le\frac{M'}{H(\xi)}$.
By \eqref{disuguglianzesuB'1}, \eqref{disuguaglianzeB''} and Young's inequality with suitable exponents, we have
\begin{equation}\label{stimaI_5}
\begin{split}
    I_4 &\le\frac{p^2}{4}M' \int_\Omega \frac{B'(H(\nabla u))}{H(\nabla u)}|\nabla \varphi|^2 |u|^{\beta+1} \varphi^{p-2} \\ & \le \frac{p^2}{4}M'C_1\int_{\Omega_1} |\nabla \varphi|^2 |u|^{\beta+1} \varphi^{p-2}+\frac{p^2}{4}M'C_2\int_{\Omega_2}H(\nabla u)^{p-2}|\nabla \varphi|^2 |u|^{\beta+1} \varphi^{p-2} \\ 
    &\le \frac{p^2}{4}M'C_1\int_{\Omega_1} |\nabla \varphi|^2|u|^{\beta+1}\varphi^{p-2}\\&+\frac{\epsilon ^2}{3}\int_{\Omega_2} B''(H(\nabla u))H(\nabla u)^{2}|u|^{\beta-1}\varphi ^p+\tilde{C}_4\int_{\Omega_2}|\nabla \varphi|^p|u|^{\beta+p-1}
    \end{split}
\end{equation}
 where $\tilde{C}_4=\tilde{C}_4(B,H,\delta,\epsilon)$ is a positive constant. 

Similarly, we have:
\begin{equation}\label{stimaI_3}
\begin{split}
    I_3&\le {\e^2}\int_{\Omega_1}H(\nabla u)^2 |u|^{\beta-1}\varphi ^p+
    \tilde{C}_3\int_{\Omega_1} |\nabla \varphi|^2|u|^{\beta+1}\varphi^{p-2}\\
    &+\frac{\e ^2}{3}\int_{\Omega_2} B''(H(\nabla u))H(\nabla u)^{2}|u|^{\beta-1}\varphi ^p+
    \tilde{C}_3\int_{\Omega_2}|\nabla \varphi|^p|u|^{\beta+p-1}
    \end{split}
\end{equation}

for a positive constant $\tilde{C}_3=\tilde{C}_3(B,H,\delta,\e,\beta)$ .

Plugging \eqref{stimaI_2},
\eqref{stimaI_5}, \eqref{stimaI_3}, 
into \eqref{stimastabile}, we get

\begin{equation}
\begin{split}
     q\int_{\Omega} |u|^{q+\beta}\varphi ^p  \le  &
     \left( \frac{\beta+1}{2}\right)^2\int_\Omega B''(H(\nabla u))H(\nabla u)^{2}|u|^{\beta-1}\varphi ^p \\
     & +\epsilon ^2 \int_{\Omega_1}H(\nabla u)^2 |u|^{\beta-1}\varphi ^p+\epsilon ^2\int_{\Omega_2} B''(H(\nabla u))H(\nabla u)^{2}|u|^{\beta-1}\varphi ^p \\&+\tilde{C}_1\int_{\Omega_1} |\nabla \varphi|^2|u|^{\beta+1}\varphi^{p-2}+\tilde{C}_1\int_{\Omega_2}|\nabla \varphi|^p|u|^{\beta+p-1},
    \end{split}
\end{equation}

where $\tilde{C}_1=\tilde{C}_1(B,H,\delta,\e,\beta)$ is a positive constant.

Now using the third estimate of  \eqref{disuguglianzesuB'1}, recalling that $\overline{C}_1=\Gamma (p-1) (\kappa+\delta)^{p-2}$, we have 
\begin{equation}\label{eqzi}
\begin{split}
     q\int_{\Omega} |u|^{q+\beta}\varphi ^p  \le &
     \max \left\{ \left(\frac{\beta+1}{2}\right)^2\Gamma (p-1) (\kappa+\delta)^{p-2}+\e^2,\left(\frac{\beta+1}{2}\right)^2+\e^2\right\} \\
     &\left(\int_{\Omega_1}H(\nabla u)^2 |u|^{\beta-1}\varphi ^p+\int_{\Omega_2} B''(H(\nabla u))H(\nabla u)^{2}|u|^{\beta-1}\varphi ^p\right) \\&+\tilde{C}_1\int_{\Omega_1} |\nabla \varphi|^2|u|^{\beta+1}\varphi^{p-2}+\tilde{C}_1\int_{\Omega_2}|\nabla \varphi|^p|u|^{\beta+p-1}.
    \end{split}
\end{equation}

Plugging \eqref{step1} into \eqref{eqzi}, we have  \begin{equation}
\alpha\int_{\Omega} |u|^{q+\beta}\varphi^p\le
\tilde{C}\int_{\Omega_1} |\nabla
\varphi|^2|u|^{\beta+1}\varphi^{p-2}+\tilde{C}\int_{\Omega_2}|\nabla
\varphi|^p |u|^{\beta+p-1}
\end{equation}
for any non-negative $\varphi\in C^1_c(\Omega)$.

We note that
\begin{equation}
\lim_{\e,\delta \rightarrow 0} \alpha(\e,\delta)=
q-\frac{\max \left\{
\left(\frac{\beta+1}{2}\right)^2\Gamma (p-1) \kappa^{p-2},\left(\frac{\beta+1}{2}\right)^2\right\}}{
\min \left\{\beta \gamma \kappa ^{p-2},\beta \frac{\gamma \Gamma^{-1}}{p-1} \right\}}\,,
\end{equation}
and, if $\Gamma\kappa^{p-2}(p-1) \ge 1$, we have 
\begin{equation}
    \left(\frac{\beta+1}{2}\right)^2\Gamma (p-1) \kappa^{p-2}\ge  \left(\frac{\beta+1}{2}\right)^2 \quad \text{and} \quad  \beta\frac{\gamma \Gamma^{-1}}{p-1} \le \beta \gamma \kappa ^{p-2}.
\end{equation}
On the other hand, if $\Gamma\kappa^{p-2}(p-1) \le 1$, we have 
\begin{equation}
    \left(\frac{\beta+1}{2}\right)^2\Gamma (p-1) \kappa^{p-2}\le  \left(\frac{\beta+1}{2}\right)^2 \quad \text{and} \quad  \beta\frac{\gamma \Gamma^{-1}}{p-1} \ge \beta \gamma \kappa ^{p-2}.
    \end{equation}
Now the fact that 
$$
q-\frac{\max \left\{
\left(\frac{\beta+1}{2}\right)^2\Gamma (p-1) \kappa^{p-2},\left(\frac{\beta+1}{2}\right)^2\right\}}
{\min \left\{\beta \gamma \kappa ^{p-2},\beta \frac{\gamma \Gamma^{-1}}{p-1} \right\}}>0
$$
follows by the assumptions:
\[
\begin{split}
&\beta \in \left[1, \frac{2q-\kappa^{-p+2}\gamma^{-1}+2\sqrt{q(q-\kappa^{-p+2}\gamma ^{-1})}}{\kappa^{-p+2}\gamma ^{-1}}\right)
\quad \text{if}\quad 0<\Gamma\kappa^{p-2}(p-1)<1\,,\\
&\beta \in \left[1, \frac{2q-\Gamma ^2\kappa^{p-2}\gamma^{-1}(p-1)^2+2\sqrt{q(q-\Gamma ^2\kappa^{p-2}\gamma ^{-1}(p-1)^2)}}{\Gamma ^2\kappa^{p-2}\gamma ^{-1}(p-1)^2}\right)
\quad \text{if}\quad \Gamma\kappa^{p-2}(p-1)\ge 1.
\end{split}
\]

\textbf{Step 3.} \emph{End of the proof.} Set $\tilde{\varphi}=\varphi^l$, with
$\varphi\in C^1_c(\Omega)$, $0\le \varphi\le 1$ and let $l$ be a positive integer.
Substituting $\tilde{\varphi}$ in \eqref{step2}, we obtain:
\begin{equation}
\int_{\Omega}
|u|^{q+\beta}\varphi ^{pl}\le \tilde C l^2\int_{\Omega_1}|\nabla \varphi|^{2}|u|^{\beta+1}\varphi ^{lp-2}+
\tilde C l^p\int_{\Omega_2}|\nabla \varphi|^{p}|u|^{\beta+p-1}\varphi ^{lp-p}\,.
\end{equation}

Using Holder's inequality with exponents $\frac{q+\beta}{\beta+1}$ and $\frac{q+\beta}{q-1}$ in the integral over $\Omega _1$, and exponents $\frac{q+\beta}{\beta+p-1}$ and $\frac{q+\beta}{q-p+1}$ in the integral over $\Omega _2$ we get:
\begin{equation}\label{equaz}
\begin{split}
    \int_\Omega |u|^{q+\beta}\varphi^{lp} \le
    &\tilde Cl^2 \left(\int_{\Omega_1}
    \left(|u|^{\beta+1}\varphi ^{lp-2}\right)^{\frac{q+\beta}{\beta+1}}\right)^{\frac{\beta+1}{q+\beta}}
    \left(\int_{\Omega_1}|\nabla \varphi|^{2\frac{q+\beta}{q-1}}\right)^{\frac{q-1}{q+\beta}}\\
    +&\tilde C l^p\left(\int_{\Omega_2}\left(|u|^{\beta+p-1}\varphi ^{lp-p}\right)^{\frac{q+\beta}{\beta+p-1}}\right)^{\frac{\beta+p-1}{q+\beta}}
    \left(\int_{\Omega_2}|\nabla \varphi| ^{p\frac{q+\beta}{q+p-1}}\right)^{\frac{q-p+1}{q+\beta}}\,.
    \end{split}
\end{equation}
By \eqref{assunzionel} we deduce that:
\begin{equation}\label{condizionekpoi}
    (lp-2)\left(\frac{q+\beta}{\beta+1}\right)\ge (lp-p)\left(\frac{q+\beta}{\beta+p-1}\right)\ge lp
\end{equation}
Moreover, recalling \eqref{condizionekpoi}, and since $0\le\varphi \le 1$, \eqref{equaz} becomes:
\begin{equation}\label{eqquasifinale}
\begin{split}
    \int_\Omega |u|^{q+\beta}\varphi ^{lp}\le &\tilde Cl^2 \left(\int_{\Omega_1}|u|^{q+\beta}\varphi ^{lp}\right)^{\frac{\beta+1}{q+\beta}}\left(\int_{\Omega_1}|\nabla \varphi| ^{2\frac {q+\beta}{q-1}}\right)^{\frac{q-1}{q+\beta}}\\+ & \tilde Cl^p\left(\int_{\Omega_2}|u|^{q+\beta}\varphi ^{lp}\right)^{\frac{\beta+p-1}{q+\beta}}\left(\int_{\Omega_2}|\nabla \varphi| ^{p\frac {q+\beta}{q-p+1}}\right)^{\frac{q-p+1}{q+\beta}}.
    \end{split}
\end{equation}
Multiplying both sides of \eqref{eqquasifinale} by
$\left(\int_{\Omega_1}|u|^{q+\beta}\varphi^{lp}\right)^{-\frac{\beta+1}{q+\beta}}$, we get:
\begin{equation}\label{neweq}
    \left(\int_\Omega |u|^{q+\beta}\varphi ^{lp}\right)^{\frac{q-1}{q+\beta}}\le \tilde Cl^2 \left(\int_{\Omega_1}|\nabla \varphi| ^{2\frac {q+\beta}{q-1}}\right)^{\frac{q-1}{q+\beta}}+ \tilde Cl^p\left(\int_{\Omega}|u|^{q+\beta}\varphi ^{lp}\right)^{\frac{p-2}{q+\beta}}\left(\int_{\Omega_2}|\nabla \varphi| ^{p\frac {q+\beta}{q-p+1}}\right)^{\frac{q-p+1}{q+\beta}}.
\end{equation}
Raising \eqref{neweq} to the exponent $\frac{q+\beta}{q-1}$,
we get 

\begin{equation}\label{zuppa}
     \int_\Omega |u|^{q+\beta}\varphi ^{pl}\le C' \int_{\Omega_1}|\nabla \varphi| ^{2\frac {q+\beta}{q-1}}+ C'\left(\int_{\Omega}|u|^{q+\beta}\varphi ^{pl}\right)^{\frac{p-2}{q-1}}\left(\int_{\Omega_2}|\nabla \varphi| ^{p\frac {q+\beta}{q-p+1}}\right)^{\frac{q-p+1}{q-1}}
\end{equation}

where $C'=C'(B,H,\e,\delta,q,\beta)$ is a positive constant. Using a weighted Young inequality with exponent $\frac{q-1}{p-2}$ and $\frac{q-1}{q-p+1}$ we have 

\begin{equation}
     \int_\Omega |u|^{q+\beta}\varphi ^{pl}\le C \int_{\Omega_1}|\nabla \varphi| ^{2\frac {q+\beta}{q-1}}+ C\int_{\Omega_2}|\nabla \varphi| ^{p\frac {q+\beta}{q-p+1}}
\end{equation}
where $C=C(B,H,\delta,\e,q,\beta)$ is a positive constant.
\endproof

The case $\kappa =0$ is simpler. Indeed, we do not need
to split $\Om$ into $\Om_1$ and $\Om_2$.

\begin{prop}\label{propk0}
Let us assume that $(h_H)$ and $(h_B)$, with $\kappa =0$ in $(h_B)$. 
Let $u$ be a weak stable solution of \eqref{eq:Euler-Lagrange}
with $p\ge 2$ and $q>\Gamma\gamma^{-1}(p-1)$. Then, for any 
\begin{equation}\label{condizionesuk}
   \beta \in \left[1, \frac{2q-\Gamma \gamma^{-1}(p-1)+2\sqrt{q(q-\Gamma \gamma ^{-1}(p-1))}}{\Gamma \gamma ^{-1}(p-1)}\right) 
\end{equation} 
and any integer
\begin{equation}\label{assunzionel2}
    l>\max\left\{ \frac{q+\beta}{q-(p-1)},2\right\}
\end{equation}
there exist a positive constant
$C=C(B,H,\e,q,\beta)$ such that

\begin{equation}\label{propk=0}
\int_{\Omega} B''(H(\nabla u))H(\nabla u)^2|u|^{\beta-1}\varphi^{lp}+|u|^{q+\beta}\varphi^{lp}\le C \int_\Omega |\nabla \varphi|^{p\frac{q+\beta}{q-(p-1)}}
\end{equation}
for every $\varphi\in C^1_c(\Omega),$ with $0\le\varphi\le 1$.
\end{prop}

\begin{proof}
\textbf{Step 1.}
\emph{For any $\beta \ge 1$ and $\epsilon >0$ small enough there exist a positive constant $\overline{C}=\overline{C}(B,H,\delta,\epsilon)>0$ such that 
\begin{equation}\label{step11}
    \left(\frac{\beta\gamma\Gamma^{-1}}{p-1}-\epsilon ^2\right)\int_\Omega B''(H(\nabla u))H(\nabla u)^2|u|^{\beta-1}\varphi ^p \le \overline{C} \int_\Omega |\nabla \varphi|^p|u|^{\beta+p-1}+\int_\Omega |u|^{q+\beta}\varphi^p 
\end{equation}
for any non negative $\varphi \in C^1_c(\Omega)$.
}

As in the proof of the Proposition \ref{propimportante}, let us consider
$
\phi =|u|^{\beta-1}u\varphi^p
\in C^1_c(\Omega)$, so that
$$\nabla \phi =p |u|^{\beta-1}u\varphi ^{p-1}\nabla\varphi+\beta |u|^{\beta-1}\varphi ^p \nabla u.$$
Using $\phi$ as test function in \eqref{debil1}, we get 

\begin{equation}\label{eq2}
    \begin{split}
&\beta\int_\Omega B'(H(\nabla u))  H(\nabla u)|u|^{\beta -1}\varphi
^p \\ &\le Mp\int_\Omega B'(H(\nabla u)) |\nabla \varphi|
|u|^{\beta}\varphi^{p-1}+ \int_{\Omega}|u|^{q+\beta} \varphi ^p.
\end{split}
\end{equation}

By  $(h_B)-(iii)$, we deduce that
\begin{equation}\label{disB'k=0}
\frac{\gamma\Gamma^{-1}}{p-1}B''(H(\nabla u))H(\nabla u)\le B'(H(\nabla u))\le \Gamma H(\nabla u)^{p-1}.
\end{equation}

By \eqref{disB'k=0} and Young's inequality with exponent $p$ and $\frac{p}{p-1}$, \eqref{eq2} becomes 

\begin{equation}\label{eq3}
    \begin{split}
   &\beta\frac{\gamma\Gamma^{-1}}{p-1}\int_\Omega B''(H(\nabla u))  H(\nabla u)^2|u|^{\beta -1}\varphi
^p\\ \le
\beta&\int_\Omega B'(H(\nabla u))  H(\nabla u)|u|^{\beta -1}\varphi
^p \\ \le Mp&\int_\Omega B'(H(\nabla u)) |\nabla \varphi|
|u|^{\beta}\varphi^{p-1}+ \int_{\Omega}|u|^{q+\beta} \varphi ^p \\ \le \epsilon ^2\gamma(p-1) &\int_\Omega H(\nabla u)^p |u|^{\beta-1}\varphi ^p+\overline{C} \int_\Omega |\nabla \varphi|^p|u|^{\beta+p-1}+\int_\Omega |u|^{q+\beta}\varphi^p \\ \le \e ^2&\int_\Omega B''(H(\nabla u))  H(\nabla u)^2|u|^{\beta -1}\varphi
^p +\overline{C} \int_\Omega |\nabla \varphi|^p|u|^{\beta+p-1}+\int_\Omega |u|^{q+\beta}\varphi^p 
\end{split}
\end{equation}

where in the last inequality we used the fact that $(p-1)\gamma t^{p-2}\le B''(t)$. 

\textbf{Step 2.} 
\emph{Set
\begin{equation}
\alpha=\alpha(\e)=q- \left(\frac{(\beta+1)^2}{4}+\e^2\right)\left(\frac{\beta\gamma}{\Gamma(p-1)}-\e^2\right)^{-1}.
   \end{equation}
There exist a positive constant
$\tilde{C}=\tilde{C}(B,H,\delta,\e,\beta)$
such that
\begin{equation}\label{step22}
\alpha\int_{\Omega} |u|^{q+\beta}\varphi^p\le
\tilde{C}\int_{\Omega}|\nabla
\varphi|^p |u|^{\beta+p-1}
\end{equation}
for any non-negative $\varphi\in C^1_c(\Omega)$.
}

To prove the claim, let us consider
$$
\psi=|u|^{\frac{\beta-1}{2}}u\varphi
^{\frac{p}{2}}\in C^1_c(\Omega)\,,
$$
so that $\nabla \psi=
\left(\frac{\beta+1}{2}\right)
\varphi^{\frac{p}{2}}|u|^{\frac{\beta-1}{2}}\nabla
u+\frac{p}{2}|u|^{\frac{\beta-1}{2}}u\varphi ^{\frac{p}{2}-1}\nabla
\varphi.$ Substituting $\psi$ in \eqref{condizionestability} and recalling \eqref{hess nullo} we obtain:
\begin{eqnarray}\label{stimastabile2}
\nonumber q\int_{\Omega} |u|^{q+\beta}\varphi ^p &\le&
  \left(\frac{\beta+1}{2}\right)^2\int_{\Omega} B''(H(\nabla u))H(\nabla u)^2|u|^{\beta-1}\varphi^p \\
\nonumber &+&\frac{p^2}{4}\int_\Omega B''(H(\nabla u))|\nabla \varphi|^2|u|^{\beta+1}\varphi^{p-2} |\nabla H(\nabla u)|^2\\
\nonumber &+&\left( \frac{\beta+1}{2}\right)p\int_\Omega B''(H(\nabla u))H(\nabla u)|\nabla \varphi||\nabla H(\nabla u)||u|^{\beta}\varphi^{p-1}\\ \nonumber &+&\frac{p^2}{4}\int_\Omega B'(H(\nabla u)) \langle D^2H(\nabla u)\nabla \varphi,\nabla \varphi\rangle |u|^{\beta+1}\varphi^{p-2} \\
&=&I_1+I_2+I_3+I_4\,,
\end{eqnarray}

Now, if $p>2$, by $(h_B)$ and Young's inequality with exponents $\frac {p}{2}$ and $\frac{p}{p-2}$ and  we
get
\begin{equation}\label{stimaI_22}
I_2\le \frac{\e ^2}{3} \int_{\Omega} B''(H(\nabla u))H(\nabla u)^{2}|u|^{\beta-1}\varphi ^p
   +\tilde{C}_2\int_{\Omega}|\nabla \varphi|^p|u|^{\beta+p-1}\,,
\end{equation}

where $\tilde{C}_2=\tilde{C}_2(B,H,\e)$ is a positive constant. For $p=2$, inequality \eqref{stimaI_22} is straightforward.

Now we estimate $I_4$. We remark that there exist a positive constant $M'$ such that  

$|D^2H(\xi)|\le \frac {M'}{H(\xi)}$. Using this fact, \eqref{disB'k=0} and Young's inequality with exponent $p$ and $\frac {p}{p-1}$ we get  

\begin{equation}\label{stimaI55}
    I_4\le \frac{\e ^2}{3} \int_{\Omega} B''(H(\nabla u))H(\nabla u)^{2}|u|^{\beta-1}\varphi ^p
   +\tilde{C}_4\int_{\Omega}|\nabla \varphi|^p|u|^{\beta+p-1}\,,
\end{equation}

with constant $\tilde{C}_4=\tilde{C}_4(B,H,\epsilon)>0.$

Similarly, we have 

\begin{equation}\label{stimaI33}
    I_3\le \frac{\e ^2}{3} \int_{\Omega} B''(H(\nabla u))H(\nabla u)^{2}|u|^{\beta-1}\varphi ^p
   +\tilde{C}_3\int_{\Omega}|\nabla \varphi|^p|u|^{\beta+p-1}\,,
\end{equation}

with constant $\tilde{C}_3=\tilde{C}_3(B,H,\epsilon,\beta)>0.$

Plugging \eqref{stimaI_22}, \eqref{stimaI55},  and \eqref{stimaI33} into \eqref{stimastabile2} we have 

\begin{equation}\label{stimafinalequasi}
    q\int_\Omega |u|^{q+\beta}\varphi ^p \le \left(\frac{(\beta+1)^2}{4}-\e^2\right) \int_{\Omega} B''(H(\nabla u))H(\nabla u)^{2}|u|^{\beta-1}\varphi ^p
   +\tilde{C}_1\int_{\Omega}|\nabla \varphi|^p|u|^{\beta+p-1}\,
\end{equation}

with $\tilde C_1:=\tilde C_2+\tilde C_3+\tilde C_4.$ Plugging \eqref{step11} into \eqref{stimafinalequasi} we deduce 

\begin{equation}
\alpha\int_{\Omega} |u|^{q+\beta}\varphi^p\le
\tilde{C}\int_{\Omega}|\nabla
\varphi|^p |u|^{\beta+p-1}
\end{equation}
where $\tilde  C=\tilde C(B,H,\e,\beta)$ is a positive constant.

Since

\begin{equation}
\lim_{\epsilon\rightarrow 0}\alpha(\e)=q- \left(\frac{(\beta+1)^2}{4}\right)\left(\frac{\beta\gamma}{\Gamma(p-1)}\right)^{-1}
   \end{equation}

and  

\begin{equation}
   \beta \in \left[1, \frac{2q-\Gamma \gamma^{-1}(p-1)+2\sqrt{q(q-\Gamma \gamma ^{-1}(p-1))}}{\Gamma \gamma ^{-1}(p-1)}\right)
\end{equation}   

we deduce that $q- \left(\frac{(\beta+1)^2}{4}\right)\left(\frac{\beta\gamma}{\Gamma(p-1)}\right)^{-1}>0$.

Finally we observe that, $\eqref{step11}$ and $\eqref{step22}$ imply

\begin{equation}\label{stimagradiente}
    \int_{\Omega} B''(H(\nabla u))H(\nabla u)^{2}|u|^{\beta-1}\varphi ^p\le \hat C \int_{\Omega}|\nabla
\varphi|^p |u|^{\beta+p-1}
\end{equation}

where $\hat C=\hat C(B,H,\epsilon,\beta)$ is a positive constant.

\textbf{Step 3.} \emph{End of the proof.} Set $\tilde \varphi =\varphi ^l$, with $\varphi \in C^1_c(\Omega)$, $0\le \varphi\le 1$ and let $l$ be a positive integer. Substituting $\tilde \varphi $ in \eqref{step22} we get

\begin{equation}\label{stimaprec}
\begin{split}
\int_{\Omega}
|u|^{q+\beta}\varphi ^{pl}&\le
\tilde Cl^p\int_{\Omega}|\nabla \varphi|^{p}|u|^{\beta+p-1}\varphi ^{lp-p}\\ &\le 
\tilde C l^p\left(\int_{\Omega}(|u|^{p+\beta-1}\varphi ^{lp-p})^{\frac{q+\beta}{\beta+p-1}}\right)^{\frac{\beta+p-1}{q+\beta}}\left(\int_{\Omega}|\nabla \varphi| ^{p\frac {q+\beta}{q-p+1}}\right)^{\frac{q-p+1}{q+\beta}}.
\end{split}
\end{equation}

Now, by \eqref{assunzionel2}, we note that $(lp-p)\left(\frac{q+\beta}{\beta+p-1}\right)\ge lp$, and for $0\le \varphi\le 1$, \eqref{stimaprec} becomes

\begin{equation}
    \int_{\Omega}
|u|^{q+\beta}\varphi ^{pl}\le \tilde Cl^p\left(\int_{\Omega}|u|^{q+\beta}\varphi ^{lp}\right)^{\frac{\beta+p-1}{q+\beta}}\left(\int_{\Omega}|\nabla \varphi| ^{p\frac {q+\beta}{q-p+1}}\right)^{\frac{q-p+1}{q+\beta}}
\end{equation}

and consequently 

\begin{equation}\label{stimafinale1}
\int_{\Omega}
|u|^{q+\beta}\varphi ^{pl}\le \tilde C l^p\int_{\Omega}|\nabla \varphi| ^{p\frac {q+\beta}{q-p+1}}.
\end{equation}

Now, substituting $\tilde \varphi$ in \eqref{stimagradiente} we obtain 

\begin{equation}
\begin{split}
    \int_{\Omega} B''(H(\nabla u))H(\nabla u)^{2}|u|^{\beta-1}\varphi ^{pl} &\le \hat Cl^p \int_{\Omega}|\nabla
\varphi|^p |u|^{\beta+p-1}\varphi ^{pl-p} \\ & \le \hat Cl^p\left(\int_{\Omega}(|u|^{p+\beta-1}\varphi ^{lp-p})^{\frac{q+\beta}{\beta+p-1}}\right)^{\frac{\beta+p-1}{q+\beta}}\left(\int_{\Omega}|\nabla \varphi| ^{p\frac {q+\beta}{q-p+1}}\right)^{\frac{q-p+1}{q+\beta}}
\end{split}
\end{equation}

and using \eqref{stimafinale1}, we get 

\begin{equation}\label{stimafinale2}
    \int_{\Omega} B''(H(\nabla u))H(\nabla u)^{2}|u|^{\beta-1}\varphi ^{pl}\le \hat C' \int_{\Omega}|\nabla \varphi| ^{p\frac {q+\beta}{q-p+1}}
\end{equation}

where $\hat C' =\hat C'(\Gamma,\gamma,p,\epsilon,\beta)$ is a positive constant.

The conclusion follows now by \eqref{stimafinale1} and \eqref{stimafinale2}.

\end{proof}
We are now ready to prove our main result.

\begin{proof}[Proof of Theorem \ref{teoremaliouville1}]

For $R>0$, let $\varphi _R\in C^{\infty}_c(\R^N)$ be a cut-off function such that
\begin{equation}\label{eq:cutoff}
            \left\{\begin{array}{llll}
                0 \leq \varphi_R \leq 1 \quad &\text{in } \R^N\\
                \varphi_R\equiv 0 \quad &\text{in } \R^N\setminus B(0,2R)\\
                \varphi_R \equiv 1 \quad &\text{in } B(0,R)\\
                |\nabla \varphi_R| \leq \frac{2}{R} &\text{in } B(0,2R) \setminus B(0,R),
            \end{array}\right.
        \end{equation}
where $B(0,R)$ denotes the open ball centered at the origin and with
radius $R$.

\noindent {Case $(i)$: $ \Gamma\kappa^{p-2}(p-1)\ge 1$.} 
From Proposition \ref{propimportante} we have:
\begin{equation}\label{Rinfinito}
     \int_{\R^N} |u|^{q+\beta}\varphi_R ^{pl}\le C R^{N-2\frac{q+\beta}{q-1}}+C R^{N-p\frac{q+\beta}{q-p+1}}.
\end{equation}

Next we prove that $N-2\frac{q+\beta}{q-1}<0$.

To this end, we denote by $\A:=\Gamma^2\kappa^{p-2}\gamma^{-1}$ and  for
$t>\A (p-1)^2$, we set
\begin{equation}
\tilde{\beta}(t):=\frac{2t-\A(p-1)^2+2\sqrt{t(t-\A(p-1)^2)}}{\A(p-1)^2}\quad \text{and}\quad g(t):=2\frac{t+\tilde{\beta}(t)}{t-1}.
\end{equation}

Therefore if $N<g(q)$, we can always choose
$\beta \in \left[1,\tilde{\beta}(q)\right)$ such that
\begin{equation}
    N-2\frac{q+\beta}{q-1}<0
\end{equation}
Observing that 
\begin{equation}
    N-2\frac{q+\beta}{q-1}>N-p\frac{q+\beta}{q-p+1},
\end{equation}
the thesis follows by letting $R\rightarrow\infty$ in \eqref{Rinfinito}.

\noindent {Case $(ii)$: $ 0<\Gamma\kappa^{p-2}(p-1)<1 $.} The proof is similar to the previous case but uses \eqref{valoribeta1}, that is
\begin{equation}
          \tilde \beta (t)=
    \frac{2t-\eta+2\sqrt{t(t-\eta)}}{\eta}
    \end{equation}

    where $\eta:=\kappa^{-p+2}\gamma ^{-1}.$

\noindent {Case $(iii)$: $\kappa =0 $.} 
From Proposition \ref{propk=0} we get
\begin{equation}\label{Rinfinito2}
   \int_{\R^N}
   B''(H(\nabla u))H(\nabla u)|u|^{\beta-1}\varphi^{lp}_R+
   |u|^{q+\beta}\varphi_R^{lp}\le C R^{N-p\frac{q+\beta}{q-(p-1)}}.
\end{equation}
As before, to conclude the proof it is enough to find $\beta$ such that $N-p\frac{q+\beta}{q-(p-1)}<0$. In this case we consider 
\begin{equation}
    \overline{\beta}(t):=\frac{2t-\C(p-1)+2\sqrt{t(t-\C(p-1))}}{\C(p-1)},
\end{equation}
where we have set $\C:= \Gamma\gamma^{-1}.$ The proof can be concluded now as in the previous two cases. 

\end{proof}

\begin{proof}[Proof of Corollary \ref{corollario}] 
The proof follows by case $(iii)$ of Theorem \ref{teoremaliouville1}.

\end{proof}

\begin{proof}[Proof of Theorem \ref{teoremaliouville2}]
Let us fix $R_0>0$ so that $K\subset B(0,R_0)$. For $R>R_0+3$, let $\psi _R \in C_c^{\infty}(\R^N)$ be a cut-off function such that  

\begin{equation}\label{eq:cutoff2}
            \left\{\begin{array}{llll}
                \psi_R\equiv 0 \quad &\text{if } |x|<R_0+1\\
                \psi_R \equiv 1 \quad &\text{if } R_0+2<|x|<R\\
                \psi_R \equiv 0 \quad &\text{if } |x|>2R\\
                |\nabla \psi_R | \leq C  \quad &\text{in } B(0,R_0+2) \setminus B(0,R_0+1)\\
                |\nabla \psi_R| \leq \frac{C}{R} &\text{in } B(0,2R) \setminus B(0,R).
            \end{array}\right.
        \end{equation}

From Proposition \ref{propk0}, substituting $\psi_R$ in \eqref{propk=0} we get 

\begin{equation}\label{cose}
    \int_{R_0+2<|x|<R} B''(H(\nabla u))H(\nabla u)^2 u ^{\beta-1}+ |u|^{q+\beta} \le C+CR^{N-p\frac{q+\beta}{q-(p-1)}}.
\end{equation}

Since $\frac{N\left(p-\frac{\Gamma}{\gamma}\right)+p}{N\frac{\Gamma}{\gamma}-p}\le \frac{N\left(p-1\right)+p}{N-p}$ by \eqref{assunzioni}, choosing $\beta=1$ it follows 
\begin{equation}\label{betauguale1}
    N-p\frac{q+\beta}{q-(p-1)}<0.
\end{equation}

For $R\rightarrow \infty$, recalling that $B(t)\le C(p,\gamma,\Gamma)B''(t)t^2 $ we have

\begin{equation}\label{usarepoh}
    B(H(\nabla u)) \in L^1(\R^N), \qquad u\in L^{q+1}(\R^N).
\end{equation}

Now, using the Pohozaev identity in the anisotropic setting, see Theorem $1.3$ in \cite{MSP} for details, we get 
\begin{equation}\label{identita}
    \frac {N}{q+1} \int_{\R^N} |u|^{q+1} =\int_{\R^N} NB(H(\nabla u))-B'(H(\nabla u))H(\nabla u).
\end{equation}

By $(h_B)$ we get that $$B(t)\le \Gamma \frac{t^p}{p}\le \frac{\Gamma}{\gamma }\frac{B'(t)t}{p},  \qquad \forall t\ge 0.$$ 
Using this inequality in \eqref{identita}, we have
\begin{equation}\label{a}
  \frac {N}{q+1} \int_{\R^N} |u|^{q+1}\le \left(\frac{N\Gamma}{p\gamma}-1\right)\int_{\R^N} B'(H(\nabla u))H(\nabla u). 
\end{equation}
Now if $N\le p\Gamma^{-1}\gamma$, the conclusion follows immediately. Suppose $N> p\Gamma^{-1}\gamma$ and let us  consider the test function $\phi =\varphi_R u$, where $\varphi_R$ is defined in \eqref{eq:cutoff}. Substituting $\phi$ in \eqref{debil1}, we get 

\begin{equation}\label{disu}
\begin{split}
    \int_{B(0,2R)}B'(H(\nabla u))H(\nabla u) \varphi_R &+\int_{B(0,2R)} B'(H(\nabla u)) \langle \nabla H(\nabla u), \nabla
\varphi_R \rangle u \,dx\,  
\\  &=\int_{B(0,2R)} |u|^{q+1}\varphi_R\,dx.
\end{split}
\end{equation}

Using Holder inequality with exponents $\left(\frac{p}{p-1}, q+1, \frac{q+1-p}{p(q+1)}\right)$ we obtain 
\begin{equation}\label{Holder}
\begin{split}
    \left | \int_{B(0,2R)} B'(H(\nabla u)) \langle \nabla H(\nabla u), \nabla
\varphi_R \rangle u \,dx\, \right| \le C \left( \int_{\R^N} H(\nabla u)^p\right )^{\frac{p-1}{p}}\left(\int_{\R^N} |u|^{q+1}\right)^{\frac{1}{q+1}}R^{\frac{N(q+1-p)}{p(q+1)}-1}
\end{split}
\end{equation}

Since \eqref{usarepoh} is in force, and  $\frac{N(q+1-p)}{p(q+1)}-1<0$ by  \eqref{betauguale1}, we can let $R\rightarrow +\infty$ in \eqref{disu} and \eqref{Holder}, to get 

\begin{equation}\label{b}
\int_{\R^N} |u|^{q+1} =\int_{\R^N}B'(H(\nabla u))H(\nabla u).
\end{equation}

Combining \eqref{a} and \eqref{b} we get

\begin{equation}
    \left(\frac{N}{q+1}-\frac{N\Gamma}{p\gamma}+1\right)\int_{\R^N} |u|^{q+1}\le 0.
\end{equation}

Recalling \eqref{assunzioni} 
$$\frac{N}{q+1}-\frac{N\Gamma}{p\gamma}+1>0$$

and therefore we get the thesis.\\

\end{proof}

\end{document}